\documentclass[11pt]{amsart}

\usepackage[letterpaper,hmargin=1.1in,vmargin=1.2in]{geometry}

\usepackage{amsmath}
\usepackage{amsfonts}
\usepackage{amssymb}
\usepackage{amsthm}

\usepackage{color}
\usepackage[pdftex]{graphicx}
\usepackage{subfigure}
\usepackage{overpic}

\usepackage[colorlinks,citecolor=blue,pagebackref=true,pdftex]{hyperref}
\usepackage{paralist}

\usepackage{caption}

\usepackage[ruled]{algorithm}
\usepackage{algpseudocode}

\newcommand\bset[2]{\bigl\{ {#1}  \,:\,  {#2} \bigr\}}
\newcommand\Bset[2]{\biggl\{ {#1}  :  {#2} \biggr\}}

\newcommand\R{\mathbb{R}}               %
\newcommand\Rsym[1]{\R_\mathsf{sym}^{#1 \times #1}}
\newcommand\Rnn{\R^{n}_{\ge 0}}

\newcommand\invSpc{\mathcal{N}}

\newcommand\affine[1]{{\mathcal{#1}}}
\renewcommand\approx[1]{\widehat{#1}}

\newcommand\defn[1]{\textbf{#1}}

\renewcommand\a{\mathbf{a}}
\newcommand\x{\mathbf{x}}
\newcommand\p{\mathbf{p}}
\newcommand\q{\mathbf{q}}
\newcommand\z{\mathbf{z}}
\renewcommand\r{\mathbf{r}}

\renewcommand\u{\mathbf{u}}
\renewcommand\b{\mathbf{b}}

\newcommand\id{\mathrm{Id}}
\newcommand\relint{\mathrm{relint}}
\renewcommand\int{\mathrm{int}\,}

\newcommand\lin{\mathrm{lin}}

\newtheorem{thm}{Theorem}[section]
\newtheorem{cor}[thm]{Corollary}
\newtheorem{lem}[thm]{Lemma}
\newtheorem{prop}[thm]{Proposition}

\theoremstyle{definition}
\newtheorem{dfn}[thm]{Definition}
\newtheorem{example}[thm]{Example}

\title[Deciding polyhedrality of spectrahedra]{%
Deciding polyhedrality of spectrahedra}

\author{Avinash Bhardwaj}
\address{Avinash Bhardwaj, Department of Industrial Engineering and Operations
Research, UC Berkeley, Berkeley, USA}
\email{avinash@ieor.berkeley.edu}

\author{Philipp Rostalski}
\address{
Philipp Rostalski, Institut f\"ur Medizinische Elektrotechnik, Universit\"at zu L\"ubeck, L\"ubeck, Germany
}
\email{rostalski@ime.uni-luebeck.de}

\author{Raman Sanyal}
\address{Raman Sanyal, Institut f\"ur Mathematik, %
Freie Universit\"at Berlin, Berlin, %
Germany}
\email{sanyal@math.fu-berlin.de}

\date{\today}

\thanks{P.~Rostalski was supported by a Feodor Lynen Scholarship of the
German Alexander von Humboldt foundation.}
\thanks{R.~Sanyal was supported by a Miller Postdoctoral Research Fellowship
at UC Berkeley and by the DFG-Collaborative Research Center,
TRR 109 ``Discretization in Geometry and Dynamics''.}



\begin{document}

\begin{abstract}
    Spectrahedra are linear sections of the cone of positive semidefinite
    matrices which, as convex bodies, generalize the class of polyhedra.  In
    this paper we investigate the problem of recognizing when a spectrahedron
    is polyhedral.  We generalize and strengthen results of Ramana (1998)
    regarding the structure of spectrahedra and we devise a normal form of
    representations of spectrahedra.  This normal form is effectively
    computable and leads to an algorithm for deciding polyhedrality.
\end{abstract}

\maketitle

\section{Introduction}\label{sec:intro}

A \defn{polyhedron} $\affine{P}$ is the intersection of the convex cone of
non-negative vectors $\Rnn$ with an affine subspace. By choosing
an affine basis for the subspace, we obtain a representation
\[
    \affine{P} \ = \ \bset{\x \in \R^{d-1}}{b_i - \mathbf{a}_i^T\x \ \ge \ 0 \,
    \text{ for } i = 1, 2, \dots, n }
\]
for some $\a_1,\a_2,\dots,\a_n \in \R^{d-1}$ and $b_1,b_2,\dots,b_n \in \R$.
Polyhedra represent the geometry underlying linear
programming~\cite{ziegler95} and, as a class of convex bodies, enjoy 
considerable interest throughout pure and applied mathematics. A proper
superclass of convex bodies that inherits many of the favorable properties of
polyhedra is the class of spectrahedra.  

A \defn{spectrahedron} $\affine{S}$ is the intersection of the convex cone of
positive semidefinite matrices with an affine subspace. Identifying the affine
subspace with $\R^{d-1}$ we write 
\begin{equation}\label{eqn:aff-spec}
    \affine{S} \ = \ \bset{\x \in \R^{d-1}}{x_1 A_1 + \cdots + x_{d-1}
    A_{d-1} + A_d \ \succeq \ 0 }
\end{equation}
where $A_1, A_2,\dots, A_d \in \R^{n \times n}$ are symmetric matrices.  Thus,
a spectrahedron is to a semidefinite program, what a polyhedron is to a linear
program. The associated map $A : \R^{d-1} \rightarrow \Rsym{n}$ given by
$A(\x) = x_1A_1 + \cdots + x_{d-1} A_{d-1} + A_d$ is called an affine
(symmetric) \defn{matrix map}. A symmetric matrix $A \in \R^{n \times n}$ is
\defn{positive semidefinite} $A \succeq 0$ if $v^T A v \ge 0$ for all $v \in
\R^n$. Hence, the set of points $\affine{S} \subseteq \R^{d-1}$ at which
$A(\x)$ is positive semidefinite is determined by a (quadratic) family of
linear inequalities
\[
    l_v(\x) \ := \ v^TA(\x)v \ = \  x_1\, v^TA_1v \ + \
    x_2\, v^TA_2v \ + \
    \cdots \ + \
    x_{d-1}\, v^TA_{d-1}v  \ + \ v^TA_dv \ \ge \  0.
\]
for $v \in \R^n$.

Spectrahedra and their projections have received considerable attention in
the geometry of semidefinite optimization~\cite{pat00}, polynomial optimization~\cite{gpt10},
and convex algebraic geometry~\cite{hv07}. To see that polyhedra are
spectrahedra, observe that a diagonal matrix is positive semidefinite if and
only if the diagonal is non-negative. Thus, we have
\[
    \affine{P} \ = \ \bset{ \x \in \R^{d-1}}{ D(\x) \succeq 0 }
\]
where $D(\x) = \mathrm{Diag}(b_1-\a^T_1\x,\dots,b_n-\a_n^T\x)$ is a diagonal
matrix map.

It is a theoretically interesting and practically relevant question to
recognize when a spectrahedron is a polyhedron.  The diagonal embedding of
$\Rnn$ into the cone of positive semidefinite matrices suggests that a
spectrahedron is a polyhedron if $A(\x)$ can be diagonalized, i.e.,
$UA(\x)U^{-1}$ is diagonal for some orthogonal matrix $U$.  By basic linear
algebra this is possible if and only if $A(\p)$ and $A(\q)$ commute for all
$\p,\q \in \R^{d-1}$. While this is certainly a sufficient condition, observe
that by Sylvester's law of inertia $\affine{S} = \bset{\x}{LA(\x)L^T \succeq
0}$ for any non-singular matrix $L$. In general, however, matrices in the
image of $LA(\x)L^T$ will not commute; see Example~\ref{ex:non-orth}.  A more
serious situation is when a polyhedron is \emph{redundantly} presented as the
intersection of a proper `big' spectrahedron and a `small' polyhedron
contained in it. 
\begin{center}
\begin{overpic}[width=8cm]%
    {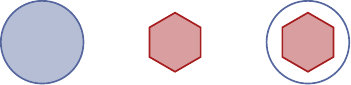}
    \put(31,11){$\bigcap$}
    \put(67,11){$=$}
    \put(6,-5){$\scriptstyle A(\x) \succeq 0$}
    \put(45,-5){$\scriptstyle B(\x) \succeq 0$}
    \put(75,-6){$\scriptstyle\left[\begin{smallmatrix}A(\x) \\ &
        B(\x)\end{smallmatrix}\right]
        \succeq 0$}

\end{overpic}
\ \\
\ \\
\end{center}

In this case, the diagonalizability criterion is genuinely lost.

In this paper we consider the question of algorithmically telling polyhedra
from spectrahedra. This question was first addressed by Ramana~\cite{ramana98}
with a focus on the computational complexity. Our results regarding the
structure of spectrahedra strengthen and generalize those of~\cite{ramana98}
and we present a simple algorithm to test if a spectrahedron $S = \bset{ x }{
A(\x) \succeq 0 }$ is a polyhedron.  The algorithm we propose consists of two
main components:

\begin{center}
\begin{tabular}{l@{\;}l@{}l}
    (Approximation)& Calculate polyhedron \ & $\affine{\approx{S}} \
    \supseteq \ \affine{S}$ from $A(\x)$, and\\
    (Containment)& determine whether & $\affine{\approx{S}} \ \subseteq \
    \affine{S}$.
\end{tabular}
\end{center}

Finding a \emph{fast} algorithm is not to be expected:  Ramana~\cite{ramana98}
showed that deciding whether a spectrahedron is polyhedral is NP-hard. As
detailed later, the `Containment' step, which is coNP-hard by the results
in~\cite{ktt12}, is done by enumerating all vertices/rays of $\approx{S}$.
This is clearly not feasible in practice and we make no claim that our
algorithm is suitable for preprocessing semidefinite programs. However, as in
the case of the `vertex enumeration problem' for polyhedra, it is of
considerable interest to have a \emph{practical} algorithm for exploration,
experimentation, and hypothesis testing with spectrahedra. Our motivation
arose in exactly this context.
We nevertheless anticipate applications of our algorithm in the area of
(combinatorial) optimization in particular in connection with
\emph{semidefinite extended formulations}\footnote{In this context it is also
of interest to detect codimension-one faces of projections of spectrahedra.
However, making statements about projections of spectrahedra is generally more
challenging as they are geometrically less well-behaved and less (algebraic)
information (such as polynomials vanishing on the boundary) is available.}; 
see, for example,~\cite{fmptw,grr12}.  In Section~\ref{sec:algorithm} our
algorithm is discussed in some detail and illustrated along an example. We
close with some remarks regarding implementation and the complexity of the
approximation step.

As for the approximation step, note that if there is a point $\p \in
\affine{S}$ with $A(\p)$ positive definite, then the \emph{algebraic
boundary}, the closure of $\partial \affine{S}$ in the Zariski topology, is
contained in the vanishing locus of $f(\x) = \det A(\x) \not\equiv 0$. Thus,
if $F \subset \affine{S}$ is a face of codimension one, the unique supporting
hyperplane is a component of the algebraic boundary of $\affine{S}$ and hence
yields a linear factor of $f$. Therefore, isolating linear factors in $f$
gives rise to a polyhedral approximation $\affine{\approx{S}}$ of
$\affine{S}$.  However, factoring a multivariate polynomial is computationally
expensive and an alternative is the use of \emph{numerical algebraic geometry}
such as Bertini~\cite{bertini} to isolate the codimension-one components of
degree one (possibly with multiplicities). Our approach avoids calculating the
determinant of the matrix map altogether by pursuing more algebro-geometric
considerations. Ramana~\cite{ramana98} showed that if $\affine{S}$ is a polyhedron,
then the relevant linear factors can be read off a block-diagonal form of
$A(\x)$. The challenge is to find the block-diagonal form.  In
Section~\ref{sec:nform} we recall and strengthen Ramana's results with very
short proofs which highlight the underlying geometry. In particular, our proof
emphasizes the role played by eigenspaces of the matrix map. From this, we
define a normal form with stronger properties and we prove that the polyhedral
approximation can be obtained by essentially computing the \emph{joint
invariant subspace} of two generic points in the image of $A(\x)$. 

{\bf Convention.}~For reasons of clarity and elegance, we will work in a
linear instead of an affine setting. That is, our main objects are exclusively
\emph{spectrahedral cones} and hence all matrix maps are \emph{linear} maps
$\R^d \rightarrow \Rsym{n}$. All results can be translated between
the linear and affine setting. The spectrahedral cone $S$ that we associate to
the spectrahedron $\affine{S}$ above is
\begin{equation}\label{eqn:lin-spec}
S \ = \ \bset{(\x,x_d) \in \R^{d}}{  
     x_1A_1 + x_2A_2 + \cdots + x_dA_d 
    \ \succeq \ 0,
                    x_d \ge 0 }.
\end{equation}
The following proposition shows that it suffices to consider spectrahedral
cones.

\begin{prop}
    The spectrahedron $\affine{S} \neq \varnothing$ given
    in~\eqref{eqn:aff-spec} is a polyhedron if and only
    if the associated spectrahedral cone $S$ given by~\eqref{eqn:lin-spec} is a polyhedral cone.
\end{prop}
\begin{proof}
    Note that for $\alpha > 0$
    \[
        (\x,\alpha) \in S \quad \Longleftrightarrow \quad \x \in \alpha
        \affine{S}.
    \]
    Indeed, $(\x,\alpha) \in S$ if and only if $\frac{1}{\alpha}\x \in
    \affine{S}$. In particular, if $S$ is a polyhedral cone, then $\affine{S}
    \cong S \cap \{(\x,x_d) : x_d = 1\}$ is a polyhedron.

    For the converse statement, we observe that the set 
    \[
        T \ := \ \{ (\x, \alpha) : \alpha > 0, \x \in \alpha \affine{S} \}
    \]
    is a subset of $S$. Since $S$ is closed, the closure $\overline{T}$ is
    contained in $S$ as well. We claim that $\overline{T} = S$. Let $(\x,0)
    \in S$. By convexity of $S$, we can pick a sequence $(\x_n,\alpha_n)_{n
    \ge 0} \in S$ with $\alpha_n > 0$ for all $n$ and $(\x_n,\alpha_n)
    \xrightarrow{n \rightarrow \infty} (\x,0)$. But this is a sequence in $T$
    and therefore $(\x,0) \in \overline{T}$. If $\affine{S}$ is a polyhedron,
    then, by the Minkowski--Weyl theorem (see~\cite[Thm.~1.2]{ziegler95}),
    $\affine{S} = \{ \x : A\x \le b \}$ for some $A \in \R^{n
    \times (d-1)}$ and $b \in \R^n$. But then 
    \[
        T \ = \ \{ (\x,x_d) : x_d > 0, A\x - x_d b \le 0 \}
    \]
    and $\overline{T} = S$ is a polyhedral cone; 
    see also~\cite[Prop.~1.14]{ziegler95}.  
\end{proof}

\textbf{Acknowledgments.} This paper grew out of a project proposed by the
last two authors for the class `Geometry of Convex Optimization' at UC
Berkeley, Fall 2010. We would like to thank Bernd Sturmfels and the
participants of the class for an inspiring environment. We would also like to
thank the referees for carefully reading the paper and their many helpful
suggestions.

\section{Normal forms and joint invariant subspaces}
\label{sec:nform}

Let $S = \bset{\x \in \R^d}{A(\x) \succeq 0}$ be a full-dimensional
spectrahedral cone given by a linear matrix map $A(\x) = x_1A_1 + x_2A_2 +
\cdots + x_dA_d$.  Throughout this section, we will assume that $A(\x)$ is of
full rank, i.e., there is a point $\p \in S$ with $A(\p) \succ 0$. As
explained in the next section, this is not a serious restriction.  We are
interested in the codimension-one faces of $S$ and how they manifest in the
presentation of $S$ given by $A(\x)$. Let us recall the characterization of
faces of a spectrahedral cone.

\begin{lem}[{\cite[Thm.~1]{RG95}}] \label{lem:face_subspace}
    Let $S = \bset{\x}{A(\x) \succeq 0 }$ be a full-dimensional
    spectrahedral cone.  For every face $F \subseteq S$ there is an
    inclusion-maximal linear subspace $\mathcal{L}_F \subset \R^n$ such that
    \[
        F \ = \ \bset{ \p \in S }{ \mathcal{L}_F \subseteq \ker A(\p) }.
    \]
\end{lem}

For the case of faces of codimension one, this characterization in terms of
kernels implies strong restrictions on the describing matrix
map.

\begin{thm}\label{thm:expose_face}
    Let $S = \bset{ \x }{ A(\x) \succeq 0}$ be a full-dimensional spectrahedral cone and let $F \subset S$ be a
    face of codimension one. Then there is a non-singular matrix $M \in
    \R^{n\times n}$ such that
    \[
        M A(\x) M^T \;=\;
        \begin{bmatrix}
            A^\prime(\x) &  \\
                  & \ell(\x) \id_k \\
        \end{bmatrix}
    \]
    where $k \ge 1$ and $\ell(\x)$ is a supporting linear form such that $F = \bset{ \x \in
    S }{ \ell(\x) = 0 }$.
\end{thm}

\begin{proof}
    Let $B = (\b_1,\b_2,\dots,\b_d)$ be a basis of $\R^d$ such that $\b_1 \in
    \int S$ and $\b_2,\dots,\b_d \in F$. By applying a suitable congruence, we
    can assume that $A(\b_1) = \id$.  In light of
    Lemma~\ref{lem:face_subspace}, let $U^T = (\u_1,\u_2,\dots,\u_n) \in
    \R^{n\times n}$ be an orthonormal basis of $\R^n$ such that
    $\mathcal{L}_F$ is spanned by $\u_{n-k+1},\dots,\u_n$ with $k = \dim
    \mathcal{L}_F$.
    It is easily seen that
    $U A(B\x) U^T$ is of the form
    \[
        \begin{bmatrix}
            A^\prime(B\x) & \\
            & x_1 \id_k
        \end{bmatrix}.
    \]
    Reverting to the original coordinates ($\x \mapsto B^{-1}\x$)
    replaces $x_1$
    by $\ell(\x)$.
\end{proof}

The form of the matrix map as given in the previous lemma expresses $S$ as the
intersection of a linear halfspace and a spectrahedral cone $S^\prime =
\bset{\x}{A^\prime(\x) \succeq 0}$. Repeating the process for $S^\prime$
proves

\begin{cor}\label{cor:polyhedral_facets}
    Let $S = \bset{\x}{A(\x)\succeq 0}$ be a full-dimensional spectrahedral
    cone. Then there is a non-singular matrix $M \in \R^{n \times n}$ such that
    \begin{equation}\label{eqn:block}\tag{$\star$}
        M A(\x) M^T \;=\;
        \begin{bmatrix}
            Q(\x) &  \\
                  & D(\x) \\
        \end{bmatrix}
    \end{equation}
    where $D(\x)$ is a diagonal matrix map of order $m \ge 0$. Moreover, if $F
    \subset S$ is a
    face of codimension one, then $F = \bset{\x \in S}{ D_{ii}(\x) = 0 }$ for
    some $1 \le i \le m$.\hfill\qed
\end{cor}

If $S$ is a polyhedral cone then all inclusion-maximal faces have codimension
one and hence $S$ is determined by $D(\x)$ alone. This recovers Ramana's
result.

\begin{cor}[{\cite[Thm.~1]{ramana98}}]\label{cor:ramana1}
    Let $S$ be a full-dimensional spectrahedral cone. Then $S$ is polyhedral
    if and only if there is a non-singular matrix  $M \in \R^{n\times n}$ such
    that
    \[
        M A(\x) M^T \;=\;
        \begin{bmatrix}
            Q(\x) &  \\
                  & D(\x) \\
        \end{bmatrix}
    \]
    where $D(\x)$ is a diagonal matrix map and $S \;=\; \bset{ \x }{ D(\x)
    \succeq 0 }$.  \hfill\qed
\end{cor}

We want to utilize Corollary~\ref{cor:polyhedral_facets} for computations but
the block-diagonal form~\eqref{eqn:block} is not canonical. This is due to the
fact that $Q(\x)$ might be further block-diagonalized giving additional linear
parts. The natural idea is to prevent this from happening.  Let us call a
matrix map $Q(\x)$ \defn{proper} if there is no $v \in \R^n$ such that $v$ is
an eigenvector of $Q(\p)$ for all $\p \in \R^d$. It is clear that if $Q(\x)$
is not proper, then there is an orthogonal matrix $U$ such that  $UQ(\x)U^t$
is block-diagonal with a block of order $1$.

\begin{dfn}
    A matrix map $A(\x)$ is in \defn{normal form} if
    \[
        A(\x) \ = \
        \begin{bmatrix}
            Q(\x) & \\
                  & D(\x)
        \end{bmatrix}
    \]
    with $Q(\x)$ proper and $D(\x)$ diagonal.
\end{dfn}

Thus for a spectrahedral cone $S$ with $A(\x)$ in normal form, we are
guaranteed to find all linear
forms defining codimension-one faces of $S$ among the linear forms in
$D(\x)$. In the rest of the section we will be concerned with the question of
how to compute the normal form. Let us start with an example where we can do
that by hand.

\begin{example}\label{ex:non-orth}
    The two dimensional spectrahedral cone given by
    \[
        A(x,y) \ = \
        x
        \begin{bmatrix}
            2 &   \\
              & 1 \\
        \end{bmatrix}
        \ + \
        y
        \begin{bmatrix}
              & 1 \\
            1 &   \\
        \end{bmatrix}
        \ \succeq \  0
    \]
    is the polyhedral cone generated by the two vectors $(1,\pm\sqrt{2})$.  A
    congruence that brings $A(x,y)$ into normal form 
    is given by
    \[
        M \ =  \
        \begin{bmatrix}
            \frac{1}{\sqrt{2}} & 1 \\
            \frac{1}{\sqrt{2}} & -1
        \end{bmatrix}.
    \]%
    The transformation $M$ is unique up to left-multiplication with
    $\mathrm{Diag}(a, b)$ and $a,b \in \R \setminus \{0\}$.  \hfill$\diamond$
\end{example}

The example shows that the congruence $M$ that brings $A(\x)$ into normal form
is not necessarily an orthogonal transformation and thus not directly related
to the eigenstructure of the matrices in the image of $A(\x)$.  It turns out
that we can assume that $M$ is orthogonal under an additional assumption.
As we will see, this is key to the computation of the normal form.   A matrix
map $A(\x)$ is \defn{unital} if $A(\p_0) = \id$ for some $\p_0 \in \R^{d}$.

\begin{prop}\label{prop:orthogonal}
    Let $A(\x)$ be a unital matrix map.  Then there is an orthogonal $n\times
    n$-matrix $U$ such that $UA(\x)U^T$ is in normal form.
\end{prop}

\begin{proof}
    If $A$ is a positive definite matrix, then, from a Cholesky decomposition,
    we get a matrix $L \in \R^{n \times n}$ such that $LAL^T = \id_n$. We call
    $L$ a Cholesky inverse of $A$.  It is unique up to left multiplication by an
    orthogonal matrix, i.e., if $L^\prime$ also satisfies the condition, then
    $(L^\prime)^{-1}L$ is orthogonal.

    Let $M$ be such that $MA(\x)M^T$ is in normal form and, since $Q(\p_0)$
    and $D(\p_0)$ are both positive definite, let $L_Q$ and $L_D$ be
    respective Cholesky inverses such that $L_D$ is diagonal. Now,
    \[
        L = \begin{bmatrix} L_Q & \\ & L_D \end{bmatrix} M
    \]
    also brings $A(\x)$ into normal form and is a Cholesky inverse for
    $A(\p_0)$. However, a Cholesky inverse
    for $A(\p_0) = \id_n$ is given by $L^\prime = \id_n$ and by the above
    remark, we see that $L = (L^\prime)^{-1}L$ is orthogonal.
\end{proof}

This result gives us a way to compute the normal form: For a unital matrix map
we seek the \defn{joint invariant subspace}, that is, the largest linear
subspace $\invSpc \subseteq \R^n$ such that for all $u \in \invSpc$ and all
$\p,\q \in \R^d$ we have $A(\p)u \in \invSpc$ (invariant subspace) and
$A(\p)A(\q)u = A(\q)A(\p)u$.  Indeed, $\invSpc$ is then the largest invariant
subspace restricted to which $A(\x)$ can be simultaneously diagonalized. This
will yield the diagonal part $D(\x)$.

At this point one could think that the joint invariant subspace of a unital
matrix map $A(\x)$ can be computed by diagonalizing a single (generic) element
$A(\p)$. That this is unfortunately not the case is the content of the
next example.

\begin{example}
    The spectrahedral cone $S$ given by
    \[
        A(x,y,t) \ = \
        \begin{bmatrix}
          t & x & y &   \\
          x & t &   &   \\
          y &   & t &   \\
            &   &   & t \\
        \end{bmatrix} \ \succeq \ 0
    \]
    is the redundant intersection of the \emph{second order cone} $\bset{
    (x,y,t)}{ t \ge 0, t^2 \ge x^2+y^2}$ and the halfspace $\bset{ (x,y,t) }{t
    \ge 0}$. 
    
    The matrix map is unital ($A(0,0,1) = \id$) and we claim that $A(x,y,t)$
    is already in normal form. Let $Q(x,y,t)$ be the principal submatrix given
    by the first three rows and columns. We need to argue that $Q(x,y,t)$ is
    proper. To this end, let $B_1 = A(1,0,0)$ and $B_2 = A(0,1,0)$. It is
    easily seen that all eigenspaces of $B_1$ and $B_2$ are one dimensional
    and that no two eigenspaces intersect non-trivially. Hence, there is no
    $v\neq 0$ that is an eigenvector for both $B_1$ and $B_2$. If $Q(x,y,t)$
    was not proper, then such a common eigenvector would exist.
    
    Up to scaling, the only eigenvector for all specializations of $A(x,y,t)$
    is $(0,0,0,1)$ with eigenvalue $\lambda = t$. But for each specialization,
    the eigenspace for $\lambda = t$ is of dimension $\ge 2$. Hence,
    $(0,0,0,1)$ is not a distinguished basis vector for the eigenspace
    corresponding to $\lambda = t$. It is therefore not possible to check if
    $A(x,y,t)$ is in normal by analyzing a (generic) point in the image.
    \hfill$\diamond$
\end{example}

The next result shows that the joint invariant subspace of $A(\x)$ can be
computed from \emph{two} generic points in the image. Here \emph{generic
points} refer to points not satisfying a certain polynomial condition (that is
implicitly given in the proof).

\begin{thm}\label{thm:calc_example} 
    Let $A(\x)$ be a unital matrix map and let $\p,\q \in \R^d$ be two
    distinct generic points. Let $\invSpc \subset \R^n$ the smallest subspace
    containing all eigenvectors common to $A(\p)$ and $A(\q)$. Then
    $\invSpc$ is invariant under any matrix in the image of $A(\x)$ and
    $\invSpc^\perp$ is the largest invariant subspace on which $A(\x)$
    restricts to a proper matrix map.
\end{thm}
\begin{proof}
    Let us assume that $A(\x)$ is already in normal form. Then the joint
    invariant subspace $\invSpc$ of $A(\x)$ can be directly read off and we
    have to show that $A(\p)$ and $A(\q)$ do not have a common eigenvector
    outside
    $\invSpc$. That is, we have to consider the situation when $Q(\p)$ and
    $Q(\q)$ have a common eigenvector.
    
    The set $\mathcal{V} \subset (\mathbb{C}^{n \times n})^2$ of pairs of
    matrices $(B_1,B_2)$ such that $B_1$ and $B_2$ have a common eigenvector
    is an algebraic variety. Hence, $\mathcal{V}$ is nowhere dense and any
    generic pair of matrices will fail to be in $\mathcal{V}$. To see that it
    is an algebraic variety, we can argue that the set of tuples
    $(B_1,\lambda_1,B_2,\lambda_2,v)$ where $v$ is an eigenvector of $B_1$ and
    $B_2$ with eigenvalue $\lambda_1$ and $\lambda_2$ respectively is clearly
    a projective algebraic variety. Using elimination theory
    (cf.~\cite[Ch.~3]{clo}) we can project onto $(B_1,B_2)$.
    The result is a proper subvariety of $(\mathbb{C}^{n \times n})^2$ that is
    equal to $\mathcal{V}$. Since $Q(\x)$ is proper, it follows that the
    image of $Q(\x)$ meets $\mathcal{V}$ in a nowhere-dense set.

    Alternatively, we can appeal to Theorem~\ref{thm:inv_spc} below: There is
    an eigenvector common to both $Q(\p)$ and $Q(\q)$ if and only if
    \[
        \bigcap_{i,j=1}^n \ker [Q(\p)^i,Q(\q)^j]  \ \neq \ \{0\}.
    \]
    Writing out this condition states that a certain matrix with entries being
    polynomials in $\p$ and $\q$ does not have full rank. This, in turn, can
    be checked by calculating a determinant which then a non-zero polynomial
    in the entries of $\p$ and $\q$. For generic $\p$ and $\q$ this
    determinant does not vanish.
\end{proof}

\section{The algorithm}
\label{sec:algorithm}

In this section we describe an algorithm for recognizing polyhedrality of a
spectrahedral cone
\[
    S \ = \ \bset{\x \in \R^d}{A(\x) \ \succeq \  0}
\]
where $A(\x)$ is a linear, symmetric matrix map of order $n$. As already
stated in the introduction, the algorithm consists of two steps: An
`approximation' step that constructs an outer polyhedral approximation
$\approx{S}$ from the matrix map $A(\x)$ that coincides with $S$ whenever $S$
is polyhedral. This is then verified in the `containment' step.

For the approximation step note that if $A(\x)$ is in normal form, then $S$ is
presented as the intersection of a spectrahedron without codimension one faces
and a polyhedron (both of which can be trivial).

\begin{prop}
    Let $S = \bset{\x \in \R^d}{A(\x) \succeq 0}$ be a full-dimensional
    spectrahedral cone with 
    \[
        A(\x) \ = \ 
        \begin{bmatrix}
            Q(\x) &  \\
                  & D(\x) \\
        \end{bmatrix}
    \]
        in normal form. Then $\approx{S} =
    \bset{\x}{D(\x) \ge 0}$ is a polyhedral cone with $S \subseteq \approx{S}$.
\end{prop}
\begin{proof} %
    Let $\p \in S$ be a point. By definition, if $A(\p)$ is
    positive semidefinite then $A(\p)_{ii} \ge 0$ for all $i$.  In
    particular, $D(\p)_{ii} \ge 0$ for all $i$ which implies that $\p \in
    \approx{S}$.
\end{proof}

Towards a procedure to bring $A(\x)$ into normal form, we need to ensure that
$S$ is full-dimensional and $A(\x)$ of full rank.
Lemma~\ref{lem:face_subspace} implies that faces of the PSD cone are
embeddings of lower-dimensional PSD cones into subspaces parametrized by
kernels. Recall that the \defn{linear hull} $\lin(C)$ of a convex cone $C$ is
the intersection of all linear spaces containing $C$ and $C$ is
full-dimensional relative to $\lin(C)$.

\begin{prop}[{\cite[Cor.~5]{RG95}}]
    Let $S = \bset{\x}{A(\x) \succeq 0}$ be a spectrahedral cone and let $\p
    \in \relint\,S$ a point in the relative interior. Then the linear hull of
    $S$ is given by
    \[
        \lin(S) \ = \ \bset{ \x \in \R^d }{ \ker A(\p) \subseteq \ker\,A(\x) }.
    \]
    If $\bar{A}(\x)$ is the restriction of $A(\x)$ to
    $(\ker\,A(\p))^\perp$, then
    \[
        S \ = \ \bset{\x \in \lin(S)}{\bar{A}(\x) \ \succeq \ 0 }
    \]
    and $\bar{A}(\p) \succ 0$.
\end{prop}

In concrete terms this means that if $M$ is a basis for the kernel of $A(\p)$
at a relative interior point $\p \in S$, then $\lin(S)$ is the kernel for all
points in the image of $MA(\x)M^T$. The map $\bar{A}(\x)$ is given by $M_0 A(\x) M_0^T$ up to a
choice of basis $M_0$ for the orthogonal complement of $\ker\,A(\p)$. Since
$\bar{A}(\p)$ is positive definite, we can choose $M_0$ so that $\bar{A}(\p) =
\id$ and hence is unital. This, for example, can be achieved by taking
advantage of the Cholesky decomposition. By choosing a basis $B$ for 
$\lin(S)$, we identify $\lin(S) \cong \R^k$ for $k=\dim S$ which insures that
$S \subset \R^k$ is full-dimensional.
The resulting spectrahedral cone
\[
    \bar{S} \ = \ \bset{ \z \in \R^k}{ \bar{A}(B\z) \ \succeq \ 0 }
\]
is linearly isomorphic to $S$ (via $B$).

In actual computations, a point in the relative interior of $S$ may be found
by interior point algorithms. In case the spectrahedral cone $S$ is
\emph{strictly feasible}, i.e., a point $\p\in \R^d$ with $A(\p) \succ 0$
exists, an interior point algorithm finds a point arbitrarily close to the
analytic center of a suitable dehomogenization of $S$. Viewed as a linear
section of the cone of positive semidefinite matrices, $S$ is not strictly
feasible, if the linear subspace only meets the boundary of $\{X \succeq 0\}$.
These are subtle but well-studied cases in which techniques from semidefinite
and cone programming such as self-dual embeddings~\cite[Ch.  5]{WoSaVa00},
facial reduction~\cite{Borwein1980}, or an iterative procedure analogous
to~\cite[Remark. 4.15]{LLR07} can be used to obtain a point $\p \in
\relint\,S$. Independent of the chosen strategy, the computation of a point
$\p \in \relint\,S$ is potentially numerically delicate and has to be handled
with care.  For the purpose of this paper, we will simply follow the first
approach as detailed in the implementation remarks below. After applying the
above procedure and possibly after a change of basis and a transformation of
the matrix map $A(\x)$ we may assume that the spectrahedral cone is indeed
full-dimensional and described by a unital matrix map.

Utilizing Theorem~\ref{thm:calc_example}, we compute the normal form of
the unital matrix map $A(\x)$ by determining an orthonormal basis for the
joint invariant subspace $\invSpc$.  The joint invariant subspace is given as
the smallest subspace containing all eigenvectors common to matrices $A(\p)$
and $A(\q)$ for generically chosen $\p,\q\in \R^d$. It can be computed either
by pairwise intersecting eigenspaces of $A(\p)$ and $A(\q)$ or, somewhat more
elegantly, by employing the following result followed by a diagonalization
step.

\begin{thm}[{\cite[Thm.~3.1]{shemesh84}}] \label{thm:inv_spc}
    Let  $A$ and $B$ be two symmetric matrices. Then the smallest
    subspace containing all common eigenvectors is given by
    \[
        \invSpc \ = \ \bigcap_{i,j = 1}^{n-1} \ker\,[A^i,B^j].
    \]
    where $[A,B] = AB - BA$ is the commutator.
\end{thm}

These techniques originate from the theory of finite dimensional $C*$-algebras
and have been used in block-diagonalizations of semidefinite programs; see
\cite{deKDP09,Kojima2010}.  After all $(n-1)^2$ commutators have been
computed, the intersection of their kernels can be computed effectively by
means of simple linear algebra.  By Theorem~\ref{thm:calc_example}, the
restriction of $A(\x)$ to $\invSpc$ is a map of pairwise commuting matrices,
there is an orthogonal transformation $U$ such that
\[
    U A(\x) U^T = \begin{bmatrix}Q(\x)& \\ &D(\x)\end{bmatrix}
\]
has the desired normal form with $Q(\x)$ proper and $D(\x)$ diagonal. The
outer polyhedral approximation of $S$ obtained from $A(\x)$ is given by
\[
    \approx{S} \ = \ \bset{ \x \in \R^d }{ D(\x) \ge 0 }.
\]
It remains to check that $\approx{S} \subseteq S$.  While deciding containment
of general (spectrahedral) cones is difficult, we exploit here the finite
generation of polyhedral cones.

\begin{thm}[{\cite[Thm.~1.3]{ziegler95}}] For every polyhedral cone
    $C$ there is a finite set $R = R(C) \subseteq C$
    such that
    \[
    C \ = \ \Bset{ \sum_{\r \in R} \lambda_\r \r }{ \lambda_\r \ge 0 \text{ for
    all } \r \in R }.
    \]
\end{thm}

Thus, if $R(\approx{S}) \subseteq S$, we infer that $\approx{S} \subseteq S
\subseteq \approx{S}$ and hence $S$ is polyhedral. Let us remark that
computationally expensive polyhedral computations may be avoided by inspecting
the lineality spaces of $S$ and $\approx{S}$ first.  The lineality space of
$S$, i.e.\ the largest linear subspaces contained in $S$, is given by
by the kernel of the linear map $A(\x)$. The complete procedure is
given in Algorithm~\ref{alg::preprocessing}. As a certificate the algorithm returns
the collection of generators $R(\approx{S})$. As we assume that $A(\x)$ is in
normal form, is can be easily checked if $S$ is polyhedral or not.

\begin{algorithm*}[htb]
    \caption{\emph{Recognizing polyhedrality of a spectrahedral cone}}
    \begin{algorithmic}[1]
        \Require Spectrahedral cone $S = \bset{\x \in \R^d}{A(\x)\succeq 0}$
        given by a linear matrix map $A(\x)$.

        \State Generate  point $\a \in \R^d$ in the relative interior of $S$.

        \State Compute unital matrix map $\bar{A}(\z)$ of order $m$ and
        linear isomorphism $B$ such that
        \[
            S \ = \ \bset{B\z}{\bar{A}(\z) \ \succeq \ 0}.
        \]

        \State Determine the joint invariant subspace $\invSpc = \bigcap_{i,j
        = 1}^{n-1} \ker\,[\bar{A}(\p)^i,\bar{A}(\q)^j]$ for two generic points $\p,\q \in
        \R^k$.

        \State Compute an orthonormal basis $U$ corresponding to the
        decomposition $\R^k
        = \invSpc^\perp \oplus \invSpc$ and compute
        \[
        U\bar{A}(\z)U^T \ = \begin{bmatrix}
            Q(\z) & \\
            & D^\prime(\z) \\
        \end{bmatrix}.
        \]

        \State Obtain diagonal map $D(\z) = V D^\prime(\z) V^T$ via an
        orthogonal transformation that diagonalizes 
        $D^\prime(\p)D^\prime(\q)$.

        \State Compute the extreme rays $R = R(\approx{S})$ of the polyhedral cone
        \[
            \approx{S} \ = \ \bset{ \z \in \R^k }{ D(\z)_{ii} \ge 0 \text{ for
            all } i = 1,\dots, \dim \invSpc}
        \]

        \State $S$ is polyhedral if and only if  $Q(\r) \succeq 0$ for all $\r
        \in R$.
\end{algorithmic}
\label{alg::preprocessing}
\end{algorithm*}

\subsection*{Implementation details.} The algorithm is implemented in Matlab
using the free optimization package Yalmip~\cite{yalmip} and is available as
part of the convex algebraic geometry toolbox 
\emph{Bermeja}~\cite{bermeja}. The SDP solver chosen for the computation of an
interior point is SeDuMi~\cite{St99}, which implements a self dual embedding
strategy and is thus guaranteed to find a point in the relative interior, even
if the spectrahedral cone is not full-dimensional. Extreme rays of $\approx{S}$
are computed using the software \texttt{cdd/cddplus}~\cite{cdd}.

In order to illustrate the algorithm, we consider the following example
involving a variant of the elliptope $\mathcal{E}_3$ (also known as the
``Samosa''), cf.~\cite{lp96}.

\begin{example}
The spectrahedral cone $S = \{\x \in \R^4 : A(\x) \succeq 0\}$ with
\begin{align*}
A(\x) = \small \begin{bmatrix} 4x_4  & 2x_4+2 x_1 &  2x_4  &  0  & 2 x_3\\
       2x_4+2 x_1 & 2x_4+2 x_1 & x_4+ x_1 &  0  &  x_3+ x_2\\
        2x_4   & x_4+ x_1  & 2x_4+ x_1 &  x_3- x_2 &   x_3\\
        0   &   0  &  x_3- x_2 & x_4+ x_1 &  0\\
       2 x_3   &  x_3+ x_2  &   x_3  &  0  &  x_4\end{bmatrix}.
\end{align*}
is to be analyzed.  Since the spectrahedral cone in context is
full-dimensional and $A(\x)$ is of full rank, i.e.\ $A(\p) \succ 0$ with
$\p=(0,0,0,1)$, the algorithm proceeds by first making the matrix
map unital. This is facilitated by applying the Cholesky inverse, computed at
the interior point $\p$. The congruence transformation $U$, thus obtained
yields the unital matrix map $\bar{A}(\z)$, allowing the use of orthogonal
transformations thereafter. 

The next step involves separating the invariant subspace from its orthogonal
complement. This step is carried out using Theorem~\ref{thm:inv_spc}, by means
of computing all commutator matrices and then intersecting their kernel. The
following step involves (simultaneous) diagonalization of the commuting part
of the matrix (here the lower right $2\times 2$ block) in order to arrive at
the desired normal form. This transformation matrix $V$ may be computed by
diagonalizing any generic matrix in the image, restricted to the commuting
part. The corresponding unital matrix map $\textstyle UA(\x)U^T$ and its
normal form $\textstyle MA(\x)M^T$ with $\textstyle M =\small
\begin{bmatrix}
 I & \\
      &V
\end{bmatrix}U$ are depicted below:
\begin{figure}[h!]
\centering
\subfigure{
\scriptsize
$
UA(\x)U^T = \begin{bmatrix}
    x_4 &  x_1 &  x_3 & 0 & 0\\
     x_1 & x_4 &  x_2 & 0 & 0\\
     x_3 &  x_2 & x_4 & 0 & 0\\
    0 & 0 & 0 & x_4+ x_1 &  x_3- x_2\\
    0 & 0 & 0 &  x_3- x_2 & x_4+ x_1
\end{bmatrix},
$
}
\subfigure{
\scriptsize
$
MA(\x)M^T =
\begin{bmatrix}
     x_4 &  x_1 &  x_3 &    0    &    0\\
      x_1 & x_4 &  x_2 &    0    &    0\\
      x_3 &  x_2 & x_4 &    0    &    0\\
     0 & 0 & 0 & x_4+ x_3- x_2+ x_1 &    0\\
     0 & 0 & 0 &    0    & x_4- x_3+ x_2+ x_1
\end{bmatrix}
$
}
\end{figure}

The normal form clearly shows that the spectrahedral cone has two polyhedral faces.
\begin{figure}[h!]
\centering
\subfigure{
  \includegraphics[scale=0.3]{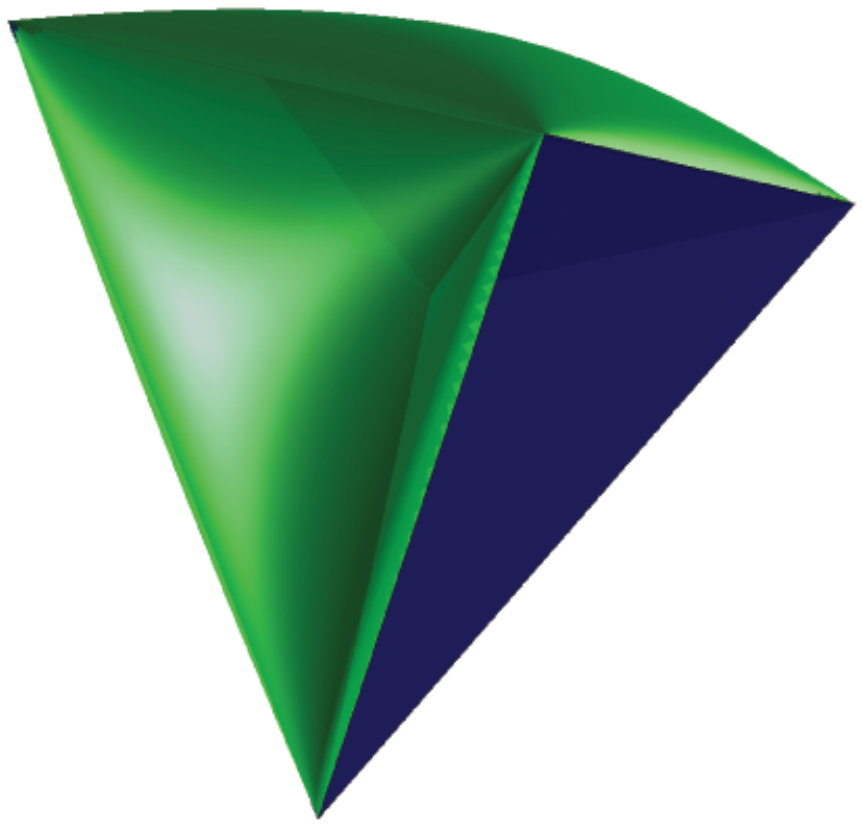}
}
\hspace{1cm}
\subfigure{
  \includegraphics[scale=0.3]{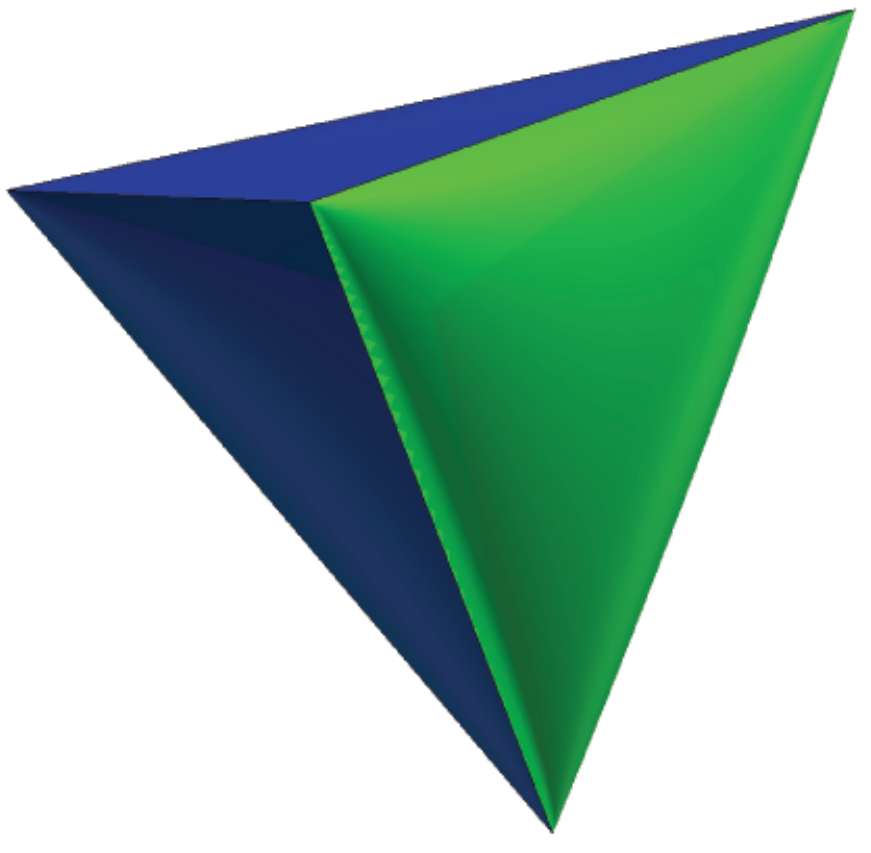}
}
\caption{Dehomogenization $\affine{S}$ (at $x_4=1$), of the spectrahedral cone $S$}\label{fig:specplot}
\end{figure}
The algorithm eventually terminates by confirming existence of a lineality
space in the corresponding polyhedral cone, even though the initial
spectrahedral cone was pointed.  This ensures the non-polyhedrality of $S$.
Figure~\ref{fig:specplot} shows a dehomogenization ($x_4 = 1$) of $S$ with its
two polyhedral facets.
\end{example}

\subsection*{A word about complexity}
Calculating the joint invariant subspace of a matrix map by way of
Theorem~\ref{thm:calc_example} requires the generation of two generic
points. In practice picking random points works very well but it is not
guaranteed to give generic points. Alternatively Theorem~\ref{thm:inv_spc}
can be used to compute the joint invariant subspace $\invSpc_{ij}$ of $A_i$
and $A_j$ for all $i < j$ and then take $\invSpc = \bigcap_{ij} \invSpc_{ij}$.
Either way, calculating the joint invariant subspace of a matrix map can be
done in polynomial time. The transformation of $A(\x)$ to an unital matrix map
is more involved.  The following example, adapted from
\cite[Example~23]{ramana97}, shows that any such procedure may involve numbers
with doubly-exponential bit complexity.

\begin{example}\label{ex:complexity}
    Consider the family of spectrahedral cones
    \[
        S_i \ = \
            \Bset{\x \in \R^{d+1}}{
                \begin{bmatrix}
                    x_{i+1} & 2 x_i \\
                    2 x_i & x_{0}
                \end{bmatrix} \ \succeq \ 0
            } \ = \
            \Bset{\x \in \R^{d+1}}{
            \begin{array}{c@{\;\;\ge\;\;}l}
                x_0 & 0 \\
                x_0x_{i+1} & 4x_i^2
            \end{array} }
    \]
    for $i = 0,\dots,d-1$.  The intersection $S = S_0 \cap S_1 \cap \cdots
    \cap S_{d-1}$ is strictly contained in the cone $\{\x \in \R^{d+1} : x_i
    \ge 2^{2^{i-1}} x_0\}$.  Denote by $A(\x)$ the matrix map for $S$.  Now
    assume that $B(\x)$ is a matrix map for $S$ such that $B(\p) = \id$ for
    some $\p \in \int S$.  Then $B(\x) = UL\,A(\x)\,(UL)^T$ where $L$ is the
    Cholesky inverse of $A(\p)$ and $U$ is an orthogonal matrix.  Denote by
    $\mathbf{l} = (QL)_1$ the first column of $QL$.  From the definition of
    the Cholesky decomposition we infer that $0 < \|\mathbf{l} \|^2 = L^2_{11}
    = \frac{1}{\p_n}$ has doubly-exponential bit complexity and hence for $\q
    = (1,0,\dots,0)$, we have that $B(\q) = \mathbf{l}\mathbf{l}^T$ has
    doubly-exponential bit complexity.\hfill$\diamond$
\end{example}

We currently do not know if in the computation of the normal form, the
unital matrix map can be avoided.

\bibliographystyle{siam}
\bibliography{SpectrahedraAndPolyhedra}

\end{document}